\documentclass[10pt]{article}
\usepackage{amsmath,amssymb}
\usepackage{enumerate}

\allowdisplaybreaks[1]

\newtheorem{theorem}{Theorem}[section]

\newtheorem{lemma}[theorem]{Lemma}

\newtheorem{definition}[theorem]{Definition}
\newtheorem{remark}[theorem]{Remark}

\newcommand{\Z}{\mathbb{Z}}

\renewcommand{\ker}{\operatorname{Ker}}
\newcommand{\id}{\operatorname{id}}

\newcommand{\Sym}{\operatorname{Sym}}
\newcommand{\aut}{\operatorname{Aut}}

\newcommand{\soc}{\operatorname{Soc}}

\newcommand{\Aut}{\operatorname{Aut}}

\newcommand{\gr}{\operatorname{gr}}

\newenvironment{proof}{\par\noindent{\bf Proof.}}{$\qed$\par\bigskip}
\newcommand{\qed}{\enspace\vrule  height6pt  width4pt  depth2pt}
\usepackage{color}

\begin{document}
\title{Primitive set-theoretic solutions of the Yang-Baxter equation\thanks{The first author was partially
supported by the grants MINECO-FEDER  MTM2017-83487-P and AGAUR
2017SGR1725 (Spain). The second author is supported in part by
Onderzoeksraad of Vrije Universiteit Brussel and Fonds voor
Wetenschappelijk Onderzoek (Belgium). The third author is supported
by the National Science Centre  grant. 2016/23/B/ST1/01045 (Poland).
2010 MSC: Primary 16T25, 20B15, 20F16. Keywords: Yang-Baxter
equation, set-theoretic solution, primitive group, brace.} }
\author{F. Ced\'o \and E. Jespers \and J. Okni\'{n}ski}
\date{}

\maketitle

\begin{abstract}
To every involutive non-degenerate set-theoretic solution $(X,r)$
of the Yang-Baxter equation on a finite set $X$ there is a
naturally associated finite solvable permutation group ${\mathcal
G}(X,r)$ acting on $X$. We prove that every primitive permutation
group of this type is of prime order $p$. Moreover, $(X,r)$ is
then a so called permutation solution determined by a cycle of
length $p$. This solves a problem recently asked by A.
Ballester-Bolinches. The result opens a new perspective on a
possible approach to the classification problem of all involutive
non-degenerate set-theoretic solutions.
\end{abstract}

\section{Introduction}

A fundamental open problem is to construct  all  solutions of the Yang-Baxter equation
  $$R_{12}R_{23}R_{12} = R_{23}R_{12}R_{23},$$
where $V$ is a vector space and $R:V\otimes V \rightarrow V\otimes
V$ is a bijective linear transformation. Here $R_{ij}$ denotes the
map $V\otimes V\otimes V \rightarrow V\otimes V\otimes V$ acting as
$R$ on the $(i,j)$ tensor factors and as the identity on the
remaining factor. In the context of quantum groups and Hopf algebras
such solutions play a fundamental role and are often referred to as
R-matrices (see for example \cite{BrownGoodearl,K}). Drinfeld in
\cite{drinfeld} suggested the study of the set-theoretic solutions
of the Yang-Baxter equation, these are the bijective maps $r : X
\times  X \rightarrow X \times X$, defined for a nonempty set $X$,
such that
 $$r_{12}r_{23}r_{12} = r_{23}r_{12}r_{23}$$
where  $r_{ij}$ denotes the map $X \times X \times X \rightarrow X
\times X \times X$ acting as $r$ on the $(i,j)$ components and as
the identity on the remaining component. In their fundamental
papers, Gateva-Ivanova and Van den Bergh \cite{GIVdB}, and Etingof,
Schedler and Soloviev \cite{ESS} introduced a subclass of the
set-theoretic of solutions, the involutive non-degenerate solutions.
Recall that a set-theoretic solution $r : X \times  X \rightarrow  X
\times  X$ of the Yang-Baxter equation, written in the form $r(x,y)
= (\sigma_x(y),\gamma_y(x))$, for $x,y \in X$, is involutive if
$r^{2} = \id$, and it is non-degenerate if $\sigma_x$ and $\gamma_x$
are bijective maps from $X$ to $X$, for all $ x \in X$.  Actually,
\cite{ESS} and \cite{GIVdB} initiated new perspectives for
developing and applying a variety of algebraic methods in this
context. In particular, they introduced the associated structure
groups  and quadratic algebras with very nice structural and
arithmetical properties when $X$ is finite. More recently, a new
tool was introduced by Rump in \cite{R07}, based on algebraic
structures called (left) braces. This attracted a lot of attention
(including contributions of Bachiller, Brzezi\'nski, Catino,
Gateva-Ivanova, Rump, Smoktunowicz, Vendramin, see for example
\cite{Brz19,DS, GI18, Smok} and the references in these papers) and
lead to several breakthroughs in the area. In particular it resulted
in several new constructions of classes of solutions and it allowed
to solve some open problems \cite{B16,CCS, CJOComm, CJOabund, Smok}.
Furthermore, many new deep and intriguing connections to a variety
of other areas were discovered \cite{CedoSurvey, RumpSurvey}.

Remaining  and  challenging fundamental problems  in the area
of involutive non-degenerate set-theoretic solutions of the
Yang-Baxter equations are: constructing new classes of solutions and
a program to classify all solutions. One of the very fruitful
approaches in this context is based on two notions: first,
indecomposable solutions and, second, retractable solutions
\cite{ESS}. In both cases, the idea behind these notions is to show
that several classes of solutions come from solutions of smaller
cardinality. And therefore, the main effort should be to
characterize the minimal `building blocks' and to describe the way
how they can be used to construct arbitrary solutions.

In this direction, a fundamental result of Rump \cite{Rump1} shows
that all finite square-free involutive  non-degenerate set-theoretic
solutions $(X,r)$, with $|X|>1$, are decomposable (into smaller
solutions that are also square-free). However, it turned out that
this is no longer true in full generality. The second approach, via
the retract relation, allowed to introduce the notion of
multipermutation solutions and to define a measure of their
complexity - the multipermutation level. Certain positive results in
this direction were obtained in \cite{CJO} and later in \cite{BCV}
it was shown that finite multipermutation solutions are
characterized by the condition saying that the structure group of
the solution is a poly-${\mathbb Z}$-group. However, this approach
also fails in full generality because there exist solutions that are
not retractable.  Namely, an example of an indecomposable and
irretractable solution was already given in \cite{JO} (see
Example~8.2.14 in \cite{JObook}); this was the first example of a
solution whose structure group is not a poly-$\mathbb{Z}$ group.

Every  non-degenerate involutive  solution $(X,r)$ is  equipped
with a permutation group ${\mathcal G}(X,r)$ acting on the set
$X$. Recently, Ballester-Bolinches  stated (in a talk
\cite{BB} during a workshop in Oberwolfach  \cite{Ballesteros})
the question of describing all the finite primitive solutions,
i.e. those solutions with primitive permutation group.

In this paper we prove that every finite primitive solution $(X,r)$
is of prime order, i.e. $|X|$ is a prime number, and it is known
that all such solutions admit a very simple description. Namely,
they are permutation solutions determined by a cyclic permutation of
length $p$.

This opens a new perspective on the classification problem of all
solutions as the notion of a primitive solution is much better than
that of an indecomposable solution, because, roughly speaking, this
shows that every solution $(X,r)$ which is not of the form described
in Theorem~\ref{main} is built on an information coming from its
imprimitivity blocks, which are sets of smaller cardinality.

The paper is organized as follows. We start in
Section~\ref{prelim} with a translation of the problem to a
certain brace associated to the given solution $(X,r)$. Then, in
Section~\ref{primit} we prove our main result. The success of our
approach is thus another instance of the phenomenon showing that
nontrivial connections with the theory of braces allow to attack
problems in the area of the Yang-Baxter equation.

\section{Preliminaries} \label{prelim}
Let $X$ be a non-empty set and  let  $r:X\times X \rightarrow
X\times X$ be a map. For $x,y\in X$ we put $r(x,y) =(\sigma_x (y),
\gamma_y (x))$. Recall that $(X,r)$ is an involutive,
non-degenerate, set-theoretic solution of the Yang-Baxter equation
if $r^2=\id$, all the maps $\sigma_x$ and $\gamma_y$ are bijective
maps from $X$ to itself and
  $$r_{12} r_{23} r_{12} =r_{23} r_{12} r_{23},$$
where $r_{12}=r\times \id_X$ and $r_{23}=\id_X\times r$ are maps
from $X^3$ to itself. Because $r^{2}=\id$ one easily verifies that
$\gamma_y(x)=\sigma^{-1}_{\sigma_x(y)}(x)$, for all $x,y\in X$ (see
for example \cite[Proposition~1.6]{ESS}).

\bigskip
\noindent {\bf Convention.} Throughout the paper a solution of the
YBE will mean an involutive, non-degenerate, set-theoretic solution
of the Yang-Baxter equation.
\bigskip

The proof of our result relies on the algebraic structure of the
left brace associated to a solution. Hence, we recall some essential
background. We refer the reader to \cite{CedoSurvey} for details. A
left brace is a set $B$ with two binary operations, $+$ and $\cdot$,
such that $(B,+)$ is an abelian group (the additive group of $B$),
$(B,\cdot)$ is a group (the multiplicative group of $B$), and for
every $a,b,c\in B$,
 \begin{eqnarray} \label{braceeq}
  a\cdot (b+c)+a&=&a\cdot b+a\cdot c.
 \end{eqnarray}
Note that if we denote by $0$ the neutral element of $(B,+)$ and by
$1$ the neutral element of $(B,\cdot)$, then
$$1=1\cdot (0+0)+1=1\cdot 0+1\cdot 0=0.$$
In any left brace $B$ there is an action $\lambda\colon
(B,\cdot)\rightarrow \aut(B,+)$,  called the lambda map of $B$,
defined by $\lambda(a)=\lambda_a$ and $\lambda_{a}(b)=a\cdot b-a$,
for $a,b\in B$. We shall write $a\cdot b=ab$, for all $a,b\in B$.
A trivial brace is a left brace $B$ such that $ab=a+b$, for all
$a,b\in B$, i.e. all $\lambda_a=\id$. The socle of a left brace
$B$ is
$$\soc(B)=\{ a\in B\mid ab=a+b, \mbox{ for all
}b\in B \}.$$ Note that $\soc(B)=\ker(\lambda)$, and thus it is a
normal subgroup of the multiplicative group of $B$. The solution of
the YBE associated to a left brace $B$ is $(B,r_B)$, where
$r_B(a,b)=(\lambda_a(b),\lambda_{\lambda_a(b)}^{-1}(a))$, for all
$a,b\in B$ (see \cite[Lemma~2]{CJOComm}).

A left ideal of a left brace $B$ is a subgroup $L$ of the additive
group of $B$ such that $\lambda_a(b)\in L$, for all $b\in L$ and
all $a\in B$. An ideal of a left brace $B$ is a normal subgroup
$I$ of the multiplicative group of $B$ such that $\lambda_a(b)\in
I$, for all $b\in I$ and all $a\in B$. Note that
\begin{eqnarray}\label{addmult1}
ab^{-1}&=&a-\lambda_{ab^{-1}}(b)
\end{eqnarray}
 for all $a,b\in B$, and
    \begin{eqnarray} \label{addmult2}
     &&a-b=a+\lambda_{b}(b^{-1})= a\lambda_{a^{-1}}(\lambda_b(b^{-1}))= a\lambda_{a^{-1}b}(b^{-1}),
     \end{eqnarray}
for all $a,b\in B$. Hence, every left ideal $L$ of $B$ also is a
subgroup of the multiplicative group of $B$, and every  ideal $I$ of
a left brace $B$ also is a subgroup of the additive group of $B$. It
is known that $\soc(B)$ is an ideal of the left brace $B$ (see
\cite[Proposition~7]{R07}).

It  also is well-known that the multiplicative group of a finite
left brace is solvable (see \cite[Theorem~5.2]{CedoSurvey}). Note
that if $B$ is a finite left brace, then, because each $\lambda_a\in
\Aut (B,+)$, the Sylow subgroups $B_1,\dots,B_k$ of the additive
group of $B$ are left ideals. Because of (\ref{addmult1}) and
(\ref{addmult2}) we also get that
$$B=B_1 + \cdots + B_k = B_1\cdots B_k \mbox{ and }
B_iB_j=B_i + B_j = B_j + B_i = B_jB_i.$$ Hence $\{ B_1,\dots ,B_k\}$
is a Sylow system of the solvable group $(B,\cdot)$, i.e. a complete
set of multiplicative Sylow subgroups such that any two are
permutable.

Recall that if $(X,r)$ is a solution of the YBE, with
$r(x,y)=(\sigma_x(y),\gamma_y(x))$, then its structure group
$G(X,r)=\gr(x\in X\mid xy=\sigma_x(y)\gamma_y(x),\mbox{ for all
}x,y\in X)$ has a natural structure of left brace such that
$\lambda_x(y)=\sigma_x(y)$, for all $x,y\in X$. The additive group
of $G(X,r)$ is the free abelian group with basis $X$. The
permutation group $\mathcal{G}(X,r)=\gr(\sigma_x\mid x\in X)$ of
$(X,r)$ is a subgroup of the symmetric group $\Sym_X$ on $X$.  The
map $x\mapsto \sigma_x$, from $X$ to $\mathcal{G}(X,r)$ extends to a
group homomorphism $\phi: G(X,r)\longrightarrow \mathcal{G}(X,r)$
and $\ker(\phi)=\soc(G(X,r))$. Hence there is a unique structure of
left brace on $\mathcal{G}(X,r)$ such that $\phi$ is a homomorphism
of left braces, this is the natural structure of left brace on
$\mathcal{G}(X,r)$. In particular  $\mathcal{G}(X,r)$ is a solvable
group if $X$ is finite. Note that there is a natural action of
$\mathcal{G}(X,r)$ on the additive group of $G(X,r)$ defined by
$$a\left(\sum_{i=1}^nz_ix_i\right)=\sum_{i=1}^nz_ia(x_i),$$
for all $a\in \mathcal{G}(X,r)$,  $x_i\in X$ and $z_i\in \Z$. It is
known that the map $\varphi\colon G(X,r)\longrightarrow
(G(X,r),+)\rtimes \mathcal{G}(X,r)$ defined by
$\varphi(g)=(g,\phi(g))$, for all $g\in G(X,r)$, is a monomorphism
of groups, and thus $G(X,r)$ is a regular subgroup of the holomorph
of $(G(X,r),+)$.

\begin{lemma}\label{lambda}
Let $(X,r)$ be a solution of the YBE. Then
$\lambda_g(\sigma_x)=\sigma_{g(x)}$, for all $g\in \mathcal{G}(X,r)$
and all $x\in X$.
\end{lemma}
\begin{proof}
Let $g\in \mathcal{G}(X,r)$ and let $x\in X$. There exist a
non-negative integer $s$, and  $x_1, \dots, x_s\in X$ and
$\varepsilon_1,\dots,\varepsilon_s\in\{ -1,1\}$, such that
$g=\sigma_{x_1}^{\varepsilon_1}\cdots\sigma_{x_s}^{\varepsilon_s}$.
Let $\phi\colon G(X,r)\longrightarrow \mathcal{G}(X,r)$ be the
homomorphism of left braces such that $\phi(x)=\sigma_x$, for all
$x\in X$. Now we have
\begin{eqnarray*}\lambda_g(\sigma_x)&=&\sigma_{x_1}^{\varepsilon_1}\cdots\sigma_{x_s}^{\varepsilon_s}\sigma_x-\sigma_{x_1}^{\varepsilon_1}\cdots\sigma_{x_s}^{\varepsilon_s}\;
=\;\phi(x_1^{\varepsilon_1}\cdots x_s^{\varepsilon_s}x-x_1^{\varepsilon_1}\cdots x_s^{\varepsilon_s})\\
&=&\phi(\lambda_{x_1^{\varepsilon_1}\cdots
x_s^{\varepsilon_s}}(x))\;
=\sigma_{\lambda_{x_1}^{\varepsilon_1}\cdots
\lambda_{x_s}^{\varepsilon_s}(x)}\;
=\sigma_{\sigma_{x_1}^{\varepsilon_1}\cdots \sigma_{x_s}^{\varepsilon_s}(x)}\\
&=&\sigma_{g(x)}.
\end{eqnarray*}
\end{proof}

The following definition is due to Ballester-Bolinches \cite{BB}.
\begin{definition}
A finite solution $(X,r)$ of the YBE is said to be primitive if its
permutation group $\mathcal{G}(X,r)$ acts primitively on $X$.
\end{definition}

\begin{remark}\label{primitive}
{\rm Note that if $(X,r)$ is a finite primitive solution of
cardinality $|X|>1$, then it is indecomposable because
$\mathcal{G}(X,r)$ acts transitively on $X$. By a well-known result,
that Huppert attributes to Galois \cite[II,~3.2~Satz]{H}, since
$\mathcal{G}(X,r)$ is solvable and because $\mathcal{G}(X,r)$ acts
primitively on $X$, there exist a prime $p$ and a positive integer
$m$ such that $|X|=p^m$, and $\mathcal{G}(X,r)$ has a unique
non-trivial normal abelian subgroup $N$. Furthermore, $N\cong
(\Z/(p))^m$ is the unique minimal normal subgroup of
$\mathcal{G}(X,r)$, $N$ acts transitively on $X$,
$\mathcal{G}(X,r)=N\rtimes A$, where $A$ is the stabilizer of some
$x\in X$  and  $N$ is self-centralizing in $\mathcal{G}(X,r)$.}
\end{remark}

By \cite[Theorem 2.13]{ESS}, for each prime $p$, there is a unique
indecomposable solution $(X,r)$ of the YBE of cardinality $p$. In
this case $\mathcal{G}(X,r)\cong \Z/(p)$, and $(X,r)$ is primitive.

Recall that a solution $(X,r)$ is  said to be irretractable if
$\sigma_x\neq \sigma_y$ for all distinct elements $x,y\in X$,
otherwise the solution $(X,r)$ is retractable.

The following is a result of Ballester-Bolinches \cite{BB}.

\begin{lemma}\label{BB}
Let $(X,r)$ be a primitive solution of cardinality $p^m$, for some
prime $p$ and some integer $m>1$. Then $(X,r)$ is irretractable.
\end{lemma}

\begin{proof}
Suppose that $(X,r)$ is retractable. Let $x,y\in X$ be distinct
elements such that $\sigma_x=\sigma_y$. Let $[x]=\{ y\in X\mid
\sigma_x=\sigma_y\}$. Let $g\in \mathcal{G}(X,r)$.  We shall prove
that $g[x]=[g(x)]$, i.e. $\sigma_{g(x)}=\sigma_{g(t)}$, for all
$t\in [x]$. By Lemma~\ref{lambda},
$$\sigma_{g(x)}=\lambda_g(\sigma_x)=\lambda_g(\sigma_t)=\sigma_{g(t)}.$$
Therefore, $g[x]=[g(x)]$. Thus $[x]$ is a subset of imprimitivity.
Since $(X,r)$ is a primitive solution and $|[x]|>1$, we have that
$[x]=X$. But then $\mathcal{G}(X,r)$ is cyclic of order $p^{m}$, in
contradiction with the fact that the group $\mathcal{G}(X,r)$ is
primitive.
\end{proof}

Let $S_1$, $S_2$ be two non-empty sets. Let $G_i$ be a group of
permutations of $S_i$, for $i=1,2$. Recall that $G_1$ and $G_2$ are
permutationally isomorphic if there exist a bijection
$f_1:S_1\longrightarrow S_2$ and an isomorphism
$f_2:G_1\longrightarrow G_2$ such that $f_1(g(s))=f_2(g)(f_1(s))$,
for all $s\in S_1$ and all $g\in G_1$. In this case $(f_1,f_2)$ is
called a permutational isomorphism from $G_1$ to $G_2$ (see
\cite[Definition~2.1.2]{S}).

\begin{lemma}\label{identification}
Let $(X,r)$ be a finite irretractable solution of the YBE. Consider
the permutation group $\mathcal{G}=\mathcal{G}(X,r)$ of $(X,r)$ with
its natural structure of left brace. Then, the map $\varphi
:X\longrightarrow \mathcal{G}$ defined by $\varphi(x)=\sigma_x$ for
all $x\in X$ is an injective homomorphism of solutions of the YBE
from $(X,r)$ to the solution $(\mathcal{G},r_{\mathcal{G}})$
associated to the left brace $\mathcal{G}$. Furthermore,
$\soc(\mathcal{G})=\{0\}$. Let
$\widetilde{\varphi}:\mathcal{G}\longrightarrow
\mathcal{G}(\mathcal{G},r_{\mathcal{G}})$ be the map defined by
$\widetilde{\varphi}(g)=\lambda_g$, for all $g\in \mathcal{G}$. Let
$\varphi':X\longrightarrow \varphi(X)$ be the bijection defined by
$\varphi'(x)=\sigma_x$, for all $x\in X$. Then
$(\varphi',\widetilde{\varphi})$ is a permutational isomorphism from
$\mathcal{G}$ to $\mathcal{G}(\mathcal{G},r_{\mathcal{G}})$.
\end{lemma}

\begin{proof}
Since $(X,r)$ is irretractable, it is clear that $\varphi$ is
injective.  Since $r(x,y)=(\sigma_x(y),
\sigma^{-1}_{\sigma_x(y)}(x))$, and by Lemma~\ref{lambda},
\begin{eqnarray*}
r_{\mathcal{G}}(\sigma_x,\sigma_y)
&=&(\lambda_{\sigma_x}(\sigma_y),\lambda^{-1}_{\lambda_{\sigma_x}(\sigma_y)}(\sigma_x))
\; =\; (\sigma_{\sigma_x(y)},\lambda^{-1}_{\sigma_{\sigma_x(y)}}(\sigma_x))\\
&=&(\sigma_{\sigma_x(y)},\lambda_{\sigma^{-1}_{\sigma_x(y)}}(\sigma_x))\;
=\; (\sigma_{\sigma_x(y)},\sigma_{\sigma^{-1}_{\sigma_x(y)}(x)}),
\end{eqnarray*}
we have that $\varphi$ is a homomorphism of solutions of the YBE.
Recall that the lambda map of the left brace $\mathcal{G}$ is a
homomorphism of groups from the multiplicative group of
$\mathcal{G}$ to the group $\Aut(\mathcal{G},+)$. Hence
$\widetilde{\varphi}$ is a homomorphism of groups and clearly it is
surjective. Furthermore,
$\ker(\widetilde{\varphi})=\soc(\mathcal{G})$. Let $g\in
\soc(\mathcal{G})$. We have, by Lemma~\ref{lambda},
$\sigma_x=\lambda_g(\sigma_x)=\sigma_{g(x)}$, for all $x\in X$.
Since $(X,r)$ is irretractable, $g(x)=x$, for all $x\in X$. Hence
$g=\id$, and thus $\soc(\mathcal{G})=\{0\}$. Therefore
$\widetilde{\varphi}$ is an isomorphism of groups.

By Lemma~\ref{lambda}, we have that
$$\varphi'(g(x))=\sigma_{g(x)}=\lambda_g(\sigma_x)=\widetilde{\varphi}(g)(\varphi'(x)),$$
for all $x\in X$ and all $g\in \mathcal{G}$. Therefore
$(\varphi',\widetilde{\varphi})$ is a permutational isomorphism from
$\mathcal{G}$ to $\mathcal{G}(\mathcal{G},r_{\mathcal{G}})$, and the
result is proven.
\end{proof}

The following result is well-known.
\begin{lemma}\label{decomp}
Let $B$ be a finite left brace. Let $p_1,\dots ,p_k$ be the prime
factors of $|B|$. Let $B_i$ be the Sylow $p_i$-subgroup of the
additive group of $B$ and let $b_i\in B_i$, for $i=1,\dots ,k$.
Since $B_iB_j=B_jB_i$, for $i\neq j$, there exist unique $a_{i,j}\in
B_j$ and $c_{i,j}\in B_i$ such that $b_ib_j=a_{i,j}c_{i,j}$. Then
$\lambda_{b_i}(b_j)=a_{i,j}$.
\end{lemma}

\begin{proof}
Let $i\neq j$. By \cite[Lemma~2(i)]{CJOComm},
$b_ib_j=\lambda_{b_i}(b_j)\lambda_{\lambda_{b_i}(b_j)^{-1}}(b_i)$.
Since $B_l$ is a left ideal of $B$, for all $l=1,\dots ,k$, we have
that $\lambda_{b_i}(b_j)\in B_j$ and
$\lambda_{\lambda_{b_i}(b_j)^{-1}}(b_i)\in B_i$. Since $B_i\cap
B_j=\{0\}$, we get that $a_{i,j}=\lambda_{b_i}(b_j)$.
\end{proof}

\section{Primitive solutions} \label{primit}

We are ready to prove our main result.

\begin{theorem} \label{main}
Let $(X,r)$ be a finite primitive solution of the YBE with $|X|>1$.
Then $|X|$ is prime. Furthermore, $\sigma_x=\sigma_y$, for all
$x,y\in X$, and $\sigma_x$ is a cycle of length $|X|$.
\end{theorem}
\begin{proof}
 Let $(X,r)$ be a finite primitive solution of the YBE. Suppose that
 $|X|>1$ and that $|X|$  is not prime. Let $\mathcal{G}=\mathcal{G}(X,r)$.
Then, by Remark~\ref{primitive}, $(X,r)$ is indecomposable and
$|X|=p^m$, for some prime $p$ and integer $m>1$. Furthermore,
$\mathcal{G}$ is solvable and $\mathcal{G}=N\rtimes A$, where
$N\cong (\Z/(p))^m$ is the unique minimal normal subgroup of
$\mathcal{G}$ and $A$ is the stabilizer of some $x\in X$ and a
maximal subgroup of $\mathcal{G}$. Moreover, $N$ is
self-centralizing  in $\mathcal{G}(X,r)$.

Recall that $\mathcal{G}$ has a natural structure of left brace. By
Lemma~\ref{identification}, the multiplicative group of the left
brace $\mathcal{G}$ acts primitively on $\widehat{X}=\{\sigma_x\mid
x\in X\}$, by the lambda map. Furthermore $A={\rm
{Stab}}_{\mathcal{G}} (\sigma_x)$, for some $x\in X$.

As mentioned earlier,  $\mathcal{G}=B_1 \oplus \cdots \oplus B_k$,
where $B_i$ is the Sylow $p_i$-subgroup of $(\mathcal{G},+)$, and
each $B_i$ is a left ideal of the left brace $\mathcal{G}$ (here
$p_1, \ldots , p_k$ are the prime factors of $|\mathcal{G}|$).
Furthermore, $\mathcal{G}=B_1\cdots B_k$ and $B_iB_j=B_jB_i$, for
all $i,j$.

Without loss of generality, we may assume $p=p_1$. Since the
multiplicative group of $\mathcal{G}$ is solvable,  and $B_2\cdots
B_k=B_2\oplus\cdots \oplus B_k$ is a Hall $p_1'$-subgroup of the
multiplicative group of $\mathcal{G}$, there exists $g\in
\mathcal{G}$ such that $B_2\cdots B_k\subseteq gAg^{-1}$. Since
$gAg^{-1}$ is the stabilizer of $\sigma_{g(x)}$, we may assume that
$B_2 \oplus \cdots \oplus B_k \subseteq A$. Since
$N\vartriangleleft \mathcal{G}$, $N\subseteq B_1$.

Let $2\leq i \leq k$ and $b\in B_i$. Since $N$ is a normal subgroup
of $(\mathcal{G},\cdot )$, for every $n\in N$, there exists $n'\in
N$ such that $n\cdot b=b\cdot n'$. Hence, because $N\subseteq B_1$,
by Lemma~\ref{decomp},  we have that
   $\lambda_n(b) = b$.
So $(\lambda_{n})_{|B_i}=\id$, for every $i=2, \ldots , k$ and every
$n\in N$. Since $\lambda_n \in {\rm Aut}(\mathcal{G},+)$, it follows
that
  \begin{eqnarray} \label{identity}
   (\lambda_{n})_{|B_2 \oplus \cdots \oplus B_k} &=&\id.
   \end{eqnarray}
Write $\sigma_x=x_1 + \cdots + x_k$ with each $x_i \in B_i$. As
$\lambda_a (\sigma_x) =\sigma_x$, for all $a\in A$ and because each
$B_i$ is a left ideal, we get that $\lambda_a (x_i) =x_i$ for any
$a\in A$. Hence,
 \begin{eqnarray}
  \lambda_g (x_i )&=&  x_i, \mbox{ for any } i \geq 2 \mbox{ and any } g\in \mathcal{G}. \label{*}
  \end{eqnarray}
Note that, because of (\ref{identity}), for every $n\in N$,
  $$\lambda_{n}(\sigma_x) =\lambda_{n}(x_1) +\cdots + \lambda_{n}(x_k) =\lambda_{n}(x_{1}) +x_2 +\cdots +x_k$$
Since $\mathcal{G}$ acts transitively on $\widehat{X}$,
${\mathcal G}=NA$ and $A$ stabilizes $\sigma_x$,
\begin{eqnarray*}
\widehat{X}&=&\{\lambda_g(\sigma_x)\mid g\in\mathcal{G}\}\;
=\; \{\lambda_n(\sigma_x)\mid n\in N\}\\
&=&\{\lambda_n(x_1)+x_2+\dots+x_k\mid n\in N\} .
\end{eqnarray*}
For every subset $S$ of $\mathcal{G}$, we denote by $\langle
S\rangle_{\cdot}$ the subgroup of $(\mathcal{G},\cdot)$ generated by
$S$, and  by $\langle S\rangle_+$ will denote the subgroup of
$(\mathcal{G},+)$ generated by $S$. By (\ref{addmult1}) and
(\ref{addmult2}), we have that
  $$\mathcal{G}=\langle \widehat{X}\rangle_{\cdot}=\langle \widehat{X}\rangle_{+} = \langle X_1\rangle_{+} + \cdots +\langle X_k\rangle_{+},$$
where $X_1 = \{ \lambda_{n} (x_1 ) \mid n\in N\}$, and $X_i=\{
x_i\}$, for $i=2,\dots, k$. Thus, $B_j=\langle X_j\rangle_{+}$,  for
every $j$, and
   \begin{eqnarray}
       \mathcal{G}&=&\langle X_{1} \rangle_{+} \oplus \langle x_{2} \rangle_{+}
\oplus \cdots \oplus \langle x_{k} \rangle_{+}
        \; = B_1 \oplus B_2 \cdots \oplus B_k . \label{**}
   \end{eqnarray}
Furthermore, from (\ref{*}) and (\ref{**})  we get that
\begin{equation}\label{****}
(\lambda_{g})_{|B_2 \oplus \cdots \oplus B_k} =\id \; \mbox{ for any
} g\in \mathcal{G}.
\end{equation}
Using also Lemma~\ref{decomp}, we
obtain, for $b_1\in B_1$ and $b\in B_i$ with $i\geq 2$, that there
exists $b'\in B_1$ such that
  $$b^{-1}b_1b  =b^{-1}\lambda_{b_{1}} (b) b' = b^{-1}b b'=b'\in B_1.$$
Hence, we have shown that
   $$ (B_1,\cdot ) \lhd (\mathcal{G},\cdot ).$$
Let $Z(B_1)$ be the center of $(B_1,\cdot)$.  As
$B_{1}\supseteq N$ is a nontrivial $p$-group,  we have that
$Z(B_1)$ is a normal non-trivial subgroup of $(\mathcal{G},\cdot
)$. Since $N$ is the unique minimal normal subgroup of
$(\mathcal{G},\cdot )$, We have that $N\subseteq Z(B_1)$. Since
$N$ is self-centralizing in $(\mathcal{G},\cdot )$, we have that
$N=Z(B_1)=B_1$.

So, $N=B_1$ is an ideal of the left brace $\mathcal{G}$,
$A=B_2+\cdots +B_k$ is a left ideal. Note that by (\ref{****}),
$A=B_2\oplus\cdots\oplus B_k$ is a trivial brace. Thus, using
(\ref{****}) again, we have
$$\sigma_x=x_1+\cdots +x_k= x_1 +\lambda_{x_1}(x_2 + \cdots  +x_k)
=x_1(x_2+\cdots +x_k)=x_1x_2\cdots x_k,$$  and $x_1\in N$ and
$x_2\cdots x_k\in A$. Since $A={\rm {Stab}}_{\mathcal{G}}
(\sigma_x)$, for every $a\in A$, we have
\begin{eqnarray*}
\sigma_x&=&\lambda_a(\sigma_x)\; =\; a\sigma_x-a
\; =\; ax_1x_2\cdots x_k-a \\
&=&ax_1a^{-1}ax_2\cdots x_k-a
\; =\;ax_1a^{-1}+\lambda_{ax_1a^{-1}}(ax_2\cdots x_k)-a\\
&=&ax_1a^{-1}+ax_2\cdots x_k-a \quad\mbox{ (by (\ref{****}))}\\
&=&ax_1a^{-1}+x_2\cdots x_k \quad \mbox{ (since $A$ is a trivial brace)}\\
&=&ax_1a^{-1}+\lambda_{ax_{1}a^{-1}} (x_2\cdots x_k) \quad \mbox{ ( by (\ref{****}))}\\
&=&ax_1a^{-1}x_2\cdots x_n.
\end{eqnarray*}
Hence $ax_1a^{-1}=x_1$, for all $a\in A$. Since $x_1\in N$, we get
that $\langle x_1\rangle$ is a normal subgroup of $\mathcal{G}$, and
thus $N=\langle x_1\rangle$. But $N$ is not cyclic, a contradiction.
Therefore $|X|$ is prime.

The second part of the result follows by \cite[Theorem~2.13]{ESS}.
\end{proof}

\vspace{30pt}
 \noindent \begin{tabular}{llllllll}
  F. Ced\'o && E. Jespers\\
 Departament de Matem\`atiques &&  Department of Mathematics \\
 Universitat Aut\`onoma de Barcelona &&  Vrije Universiteit Brussel \\
08193 Bellaterra (Barcelona), Spain    && Pleinlaan 2, 1050 Brussel, Belgium \\
 cedo@mat.uab.cat && Eric.Jespers@vub.be \\ \\
 J. Okni\'{n}ski && \\ Institute of
Mathematics &&
\\  Warsaw University &&\\
 Banacha 2, 02-097 Warsaw, Poland &&\\
 okninski@mimuw.edu.pl &&
\end{tabular}

\end{document}